\documentclass[12pt,reqno]{amsart}
\usepackage{fullpage}
\usepackage{times}
\usepackage{amsmath,amssymb,amsthm,url}
\usepackage[utf8]{inputenc}
\usepackage[english]{babel}
\usepackage{comment}
\usepackage{bbm}
\usepackage{enumerate}
\usepackage{bm}
\usepackage{graphicx}
\usepackage{mathrsfs}
\usepackage[colorlinks=true, pdfstartview=FitH, linkcolor=blue, citecolor=blue, urlcolor=blue]{hyperref}

\newtheorem{thm}{Theorem}[section]

\newtheorem{prop}[thm]{Proposition}
\newtheorem{cor}[thm]{Corollary}

\theoremstyle{definition}
\newtheorem{defn}[thm]{Definition}

\newtheorem{rem}[thm]{Remark}

\newcommand{\F}{\mathbb{F}}

\newcommand{\PGL}{\operatorname{PGL}}
\newcommand{\GL}{\operatorname{GL}}
\newcommand{\AG}{\operatorname{AG}}
\newcommand{\PG}{\operatorname{PG}}

\title{Erd\H{o}s--Ko--Rado theorem in Peisert-type graphs}
\author{Chi Hoi Yip}
\address{Department of Mathematics \\ University of British Columbia \\ 1984 Mathematics Road \\ Vancouver  V6T 1Z2 \\ Canada}
\email{kyleyip@math.ubc.ca}

\begin{document}

\maketitle

\begin{abstract}
The celebrated Erd\H{o}s--Ko--Rado (EKR) theorem for Paley graphs of square order states that all maximum cliques are canonical in the sense that each maximum clique arises from the subfield construction. Recently, Asgarli and Yip extended this result to Peisert graphs and other Cayley graphs which are Peisert-type graphs with nice algebraic properties on the connection set. On the other hand, there are Peisert-type graphs for which the EKR theorem fails to hold. In this paper, we show that the EKR theorem of Paley graphs extends to almost all pseudo-Paley graphs of Peisert-type. Furthermore, we establish the stability results of the same flavor.
\end{abstract}

\section{Introduction}

Throughout the paper, let $p$ be an odd prime and $q$ a power of $p$. Let $\F_q$ be the finite field with $q$ elements, $\F_q^+$ be its additive group, and $\F_q^*=\F_q \setminus \{0\}$ be its multiplicative group. 

Given an abelian group $G$ and a connection set $S \subset G \setminus \{0\}$ with $S=-S$, the {\em Cayley graph} $\operatorname{Cay}(G, S)$ is
the undirected graph whose vertices are elements of $G$, such that two vertices $g$ and $h$ are adjacent if and only if $g-h \in S$.  A {\em clique} in a graph $X$ is a subset
of vertices in which every two distinct vertices are adjacent. Note that a subset $C$ of vertices in $X=\operatorname{Cay}(G, S)$ is a clique if and only if $C-C \subset S \cup \{0\}$. This naturally connects the concept of cliques in Cayley graphs with many interesting problems from additive combinatorics. 

Next, we introduce the main object in this paper, defined in \cite{AGLY22, AY22}:

\begin{defn}[Peisert-type graphs]\label{defn:peisert-type}
Let $q$ be an odd prime power. Let $S \subset \F_{q^2}^*$ be a union of $m \leq q$ \footnote{The definition in \cite{AY22} has the stronger assumption that $m \leq \frac{q+1}{2}$. Readers are warned to keep this in mind when consulting the results in \cite{AY22}.} cosets of $\F_q^*$ in $\F_{q^2}^*$, that is,
\begin{equation}\label{coset}
S=c_1\F_q^* \cup c_2\F_q^* \cup \cdots \cup c_m \F_q^*.
\end{equation}
Then the Cayley graph $X=\operatorname{Cay}(\F_{q^2}^+, S)$ is said to be a Peisert-type graph of type $(m,q)$. 
\end{defn}

The classical Erd\H{o}s-Ko-Rado theorem \cite{EKR61} concerns the classification of maximum intersecting families of $k$-element subsets of $\{1, 2, ..., n\}$: when $n \geq 2k+1$, all maximum families are canonically intersecting (also known as stars) in the sense that all sets in the family contain a fixed element. Unsurprisingly, in many extremal problems, the only extremal configurations are given by those obvious constructions. The book \cite{GM} by Godsil and Meagher consists of many modern algebraic approaches to proving various EKR-type results.

In this paper we study the structure of maximum cliques in Peisert-type graphs; equivalently, we study the interaction between the addition operation (coming from $C-C$) and the multiplication operation (coming from those $\F_q^*$-cosets forming $S$) over finite fields. From the perspective of EKR, Peisert-type graphs are nice objects to study as they admit obvious choices of maximum cliques. Indeed, one can read from the connection set $S$ that $c_1\F_q, c_2\F_q, \ldots, c_m\F_q$ are cliques in $X$ containing $0$. Moreover, Peisert-type graphs are based on cyclotomic constructions (due to Brouwer, Wilson, and Xiang~\cite{BWX}) so that they are strongly regular, which implies that $c_1\F_q, c_2\F_q, \ldots, c_m\F_q$ are maximum cliques because of the Delsarte bound (see \cite[Theorem 7]{AGLY22} for a different proof). Thus, translates of $c_1\F_q, c_2\F_q, \ldots, c_m\F_q$ can be regarded as \emph{canonical cliques} in $X$. Since there are no other obvious maximum cliques in $X$, in view of EKR-type results, we say \emph{the EKR theorem holds for $X$}, or equivalently, $X$ has \emph{the strict-EKR property}, provided that all maximum cliques in $X$ are canonical. 

The story of the EKR theorem in our setting begins with a celebrated result by Blokhuis~\cite{B84}, who showed that the only maximum clique (containing 0,1) in the Paley graph of order $q^2$ is given by the subfield $\F_q$. Recall that the {\em Paley graph of order $q^2$} is $\operatorname{Cay}(\F_{q^2}^+,(\F_{q^2}^*)^2)$, where $(\F_{q^2}^*)^2$ is the subset of squares in $\F_{q^2}^*$. One can easily show that the subset of squares can be decomposed as the union of $\F_q^*$-cosets, and verify that Blokhuis' result is equivalent to the EKR theorem for Paley graphs; see \cite[Section 5.9]{GM} for a related discussion. Later, Sziklai~\cite{Szi99} generalized the proof of Blokhuis, and showed that if $d \mid (q+1)$ and $d \geq 3$, then the \emph{$d$-Paley graph} $\operatorname{Cay}(\F_{q^2}^+,(\F_{q^2}^*)^d)$ has the strict-EKR property, in the same spirit. We refer to \cite[Section 2.1]{AY22} for more historical discussion. Recently, Mullin \cite{NM} studied the same question for the Peisert graph $\operatorname{Cay}(\F_{q^2}^+,(\F_{q^2}^*)^4 \cup g(\F_{q^2}^*)^4)$ (where $q \equiv 3 \pmod 4$ and $g$ is a primitive root of $\F_{q^2}$) and conjectured that the EKR theorem also holds. However, she remarked (see also \cite[Remark 2.19]{AY22}) that it seems the idea of Blokhuis and Sziklai relies on the subgroup structure of the connection set (which precisely corresponds to Paley graphs and $d$-Paley graphs) and thus fails to extend to Peisert graphs. When $d$ is even, one can also ask the same question for the $d$-th power Peisert graphs (see for example \cite[Definition 2.11]{AY22}), defined similarly as $d$-Paley graphs. 

The attempt to find a unified approach to prove the EKR theorem of these graphs eventually led to the definition of Peisert-type graphs in \cite{AY22}; indeed, the above four families of graphs are special Peisert-type graphs \cite[Lemma 2.10]{AY22}. In \cite[Theorem 1.3]{AY22}, Asgarli and Yip showed that the strict-EKR property holds for a Peisert-type graph $X$ of type $(m,q)$ (with $m \leq \frac{q+1}{2}$ and $q$ sufficiently large) provided the existence of a multiplicative character such that the character sum over each subset of the connection set of $X$ admits a small cancellation. When $X$ is a Paley graph, Peisert graph, or their generalizations, the choice of such a character is obvious; thus this criterion gives rise to a new proof of the result by Blokhuis and by Sziklai, as well as resolves the conjecture by Mullin (except for the case when $p$ is small compared to $\log_p q$). However, it is not clear how to apply such a criterion to
an abstractly defined Peisert-type graph. Furthermore, given a Peisert-type graph, it seems difficult to predict if it has the strict-EKR property by staring at its connection set. Indeed, one can systematically construct an infinite family of Peisert-type graphs for which the strict-EKR property fails \cite[Section 5]{AGLY22}. 

The stability of canonical cliques in Peisert-type graphs has also been studied. For simplicity, we say a Peisert-type graph $X$ with order $q^2$ has \emph{the property $\operatorname{ST}(k)$} if each clique in $X$ with size at least $q-k$ is contained in a canonical clique in $X$. From the definition, it is clear that the strict-EKR property is equivalent to the property $\operatorname{ST}(0)$. Thus, the property $\operatorname{ST}(k)$ (for $k \geq 1$) refines the strict-EKR property. One widely open conjecture, due to by Baker, Ebert, Hemmeter, and Woldar \cite{BEHW96}, states that the second largest maximal clique in the Paley graph with order $q^2$ has size $\frac{q+r(q)}{2}$, where $r(q)=1$ or $3$, depending on  $q $ modulo $4$. In other words, Paley graphs of order $q^2$ has the property $\operatorname{ST}(\frac{q-r(q)}{2}-1)$. To the best knowledge of the author, no partial progress has been made for this conjecture. For $d$-Paley graphs of order $q^2$ with $d \geq 3$, Sziklai \cite[Theorem 1.3]{Szi99}
 showed they have the property $\operatorname{ST}(c\sqrt{q})$ for some positive constant $c$ (depending on $d$); we refer to \cite[Corollary 4.4]{AY22} and \cite[Section 6]{GSY23} for some generalization and improvement on his results. However, all known results in this direction assume the edge density of the graph being strictly less than $1/2$. In general, such type of stability questions is widely open for Peisert-type graphs.

In \cite[Section 6]{AGLY22}, Asgarli, Goryainov, Lin, and Yip asked for a complete classification of Peisert-type graphs with the strict-EKR property. This question is potentially very challenging; indeed, it turns out to be a special case of a widely open problem regarding the strict-EKR property of block graphs of orthogonal arrays \cite[Problem 16.4.1]{GM}. In this paper, we give some partial answers to this question by estimating the number of Peisert-type graphs with the strict-EKR property. In other words, we study how likely the strict-EKR property holds for a ``random" Peisert-type graph. From the probabilistic viewpoint, it seems plausible that there should be a threshold in terms of $m$ and $q$ determining the prevalence (or not) of the strict-EKR property among Peisert-type graphs of type $(m,q)$. We remark that questions of a similar flavor have been studied previously; for example, Balogh, Bohman, and Mubayi \cite{BBM09} studied EKR-type results in random hypergraphs. 

Our first main result shows that the strict-EKR property is prevalent when the edge density of the graph is at most $\frac{1}{2}$. It is known that Peisert-type graphs with edge density $\frac{1}{2}$ are pseudo-Paley graphs \cite[Corollary 5]{AGLY22}, that is, graphs that share the same spectrum with some Paley graphs. Thus, our result shows that the EKR theorem for Paley graphs extends to almost all pseudo-Paley graphs of Peisert-type; moreover, we can say something stronger on the stability of canonical cliques. 

\begin{thm}\label{thmST1}
As $q \to \infty$, almost all Peisert-type graphs of type $(\frac{q+1}{2},q)$ have the property $\operatorname{ST}(\sqrt{q}/2)$. 
\end{thm}

In many problems defined over finite fields, working over a prime field $\F_p$ is much easier than working over a general finite field. Essentially, this is due to the existence of subfield obstructions over a general finite field. The same situation occurs in our settings. When $p$ is a prime, it is known that all Peisert-type graphs of type $(\frac{p+1}{2},p)$ have the strict-EKR property \cite{AY22} (see also Theorem~\ref{subspacethm}). Our second main result refines Theorem~\ref{thmST1} when $q=p$ is a prime.

\begin{thm}\label{thmST2}
As $p \to \infty$, almost all Peisert-type graphs of type $(\lfloor \frac{2p-2}{3} \rfloor,p)$ have the strict-EKR property, and almost all Peisert-type graphs of type $(\frac{p+1}{2},p)$ have the property $\operatorname{ST}(p/20)$.
\end{thm}

On the other hand, we show that if the edge density of a Peisert-type graph is too large, then the strict-EKR property almost surely fails. 

\begin{thm}\label{thmnoEKR}
Let $n \geq 2$ with the largest proper divisor being $t$ and let $q=p^n$. As $p \to \infty$, almost all Peisert-type graphs of type $(q-o(p^t),q)$ do not have the strict-EKR property. 
\end{thm}

As suggested by numerical computations, the structure of
maximum cliques could be very complicated in Peisert-type graphs with the edge density strictly greater than $1/2$. This is probably why the strict-EKR property of these graphs has not been studied in the literature.

\medskip

\textbf{Structure of the paper.} 
In Section~\ref{sec: prelim}, we provide additional background and prove some preliminary results. In Section~\ref{sec: EKR}, we study the strict-EKR property of Peisert-type graphs. Finally, in Section~\ref{sec: ST}, we study the stability of canonical cliques. The proof of Theorem~\ref{thmST1} and Theorem~\ref{thmST2} will be presented in Section~\ref{sec: ST}, and the proof of Theorem~\ref{thmnoEKR} will be presented in Section~\ref{sec: EKR}.

\section{Preliminaries}\label{sec: prelim}

\subsection{EKR properties of Peisert-type graphs} In this subsection, we provide extra background on the known EKR properties of Peisert-type graphs.

\begin{thm}[Asgarli and Yip {\cite[Theorem 1.2]{AY22}}]  \label{subspacethm}
Let $X$ be a Peisert-type graph of type $(m,q)$ with $m \leq \frac{q+1}{2}$, where $q$ is a power of an odd prime $p$. Then each maximum clique in $X$ containing 0 is an $\F_p$-subspace of $\F_{q^2}$. 
\end{thm}  

In particular, the above theorem immediately implies the strict-EKR property for each Peisert-type graph of type $(m,p)$ with $m \leq \frac{p+1}{2}$, which has been implicitly observed by Lov\'{a}sz and Schrijver \cite{LS83}. When $q$ is a proper prime power and the connection set of $X$ has a nice algebraic structure (which is the case for Paley graphs, Peisert graphs, and their generalizations), it turns out that Theorem~\ref{subspacethm} implies the strict-EKR property \cite[Theorem 1.3]{AY22}. 

The following theorem shows that all Peisert-type graphs with small edge density have the strict-EKR property. 

\begin{thm}[Asgarli, Goryainov, Lin, and Yip {\cite[Corollary 8]{AGLY22}}]\label{OAEKR}
If $q>(m-1)^2$, then all Peisert-type graphs of type $(m,q)$ have the strict-EKR property. 
\end{thm}

Theorem~\ref{OAEKR} is best possible if $q$ is a square; see the counterexamples constructed in \cite[Theorem 9]{AGLY22}. We refer to \cite[Corollary 4.1]{AY22} for a slightly stronger statement.

The following corollary is equivalent to the above theorem. We record it here as it has potential applications in number theory and additive combinatorics. The realization of this equivalence is also helpful in proving our main results.

\begin{cor}\label{subspace_covering}
Let $V \subset \F_{q^2}$ be an $\F_p$-subspace with size $q$. If $V$ is not an $\F_q$-subspace, then at least $\sqrt{q}+1$ many $\F_q^*$-cosets are needed to cover $V \setminus \{0\}$.
\end{cor}

\begin{proof}
Suppose that the minimum number of $\F_q^*$-cosets needed to cover $V \setminus \{0\}$ is $m$; say $V \setminus \{0\} \subset S:=\cup_{i=1}^m c_i\F_q^*$. Then $V$ is a clique in $X=\operatorname{Cay}(\F_{q^2}^+,S)$ since $V-V=V \subset S \cup \{0\}$. It follows that $V$ is a non-canonical clique in the Peisert-type graph $X$ that is of type $(m,q)$, and thus $X$ fails to satisfy the strict-EKR property. It follows from Theorem~\ref{OAEKR} that $m \geq \sqrt{q}+1$.
\end{proof}

\subsection{Connection with finite geometry}
In this subsection, we recall some celebrated results from the theory of directions and then discuss their connection with Peisert-type graphs. 

We begin with some standard terminologies. Let $\AG(2,q)$ and $\PG(1,q)$ denote the affine plane and the projective line over $\F_q$, respectively. 
Let $U$ be a subset of $\AG(2,q)$; the \emph{set of directions determined} by $U$ is defined to be
$$
D (U):=  \left\{ [a-c: b-d]:  (a,b), (c,d) \in U,\, (a,b) \neq (c,d) \right\} \subset \PG(1,q).
$$

Let $D \subset \PG(1,q) $ be a set of directions. We say that a subset $U$ of $\AG(2,q)$ is a \emph{$D$-set} if $D(U)=D$, and $U$ is said to be \emph{maximal with respect to $D$} if $D(U) \subset D$ and $D(U \cup \{v\}) \not \subset D$ for any $v \in \AG(2,q) \setminus U$.

The theory of directions has been developed by multiple
authors, notably R\'edei \cite{R73} and Sz\H{o}nyi \cite{S99}. It is of particular interest to estimate $|D(U)|$; the following theorem summarizes the best-known results on the size of $D(U)$ when $q=p$ is a prime.
\begin{thm}\label{direction}
Let $U$ be a subset of $\AG(2,p)$ with $|U|=p$.
\begin{itemize}
    \item (R\'{e}dei and Megyesi \cite{R73}) If the points in $U$ are not all collinear, then $U$ determines at least $\frac{p+3}{2}$ directions.
    \item (Lov\'{a}sz and Schrijver \cite{LS83}) If $U$ determines exactly $(p+3) / 2$ directions, then $U$ is affinely equivalent to the set $\{(x,x^{(p+1)/2}): x \in \F_p\}$ . 
    \item (G\'{a}cs \cite{G03}) If $U$ determines more than $\frac{p+3}{2}$ directions, then it determines at least  $\lfloor \frac{2p+1}{3} \rfloor$ directions.
\end{itemize}
\end{thm}

The following theorem is analogous to Theorem~\ref{direction} for $q=p^2$.

\begin{thm}[G\'{a}cs, Lov\'{a}sz, and Sz\H{o}nyi \cite{GLS07}]\label{p^2}
Let $U$ be a subset of $\AG(2,p^2)$ with $|U|=p^2$. If $|D(U)| \geq \frac{p^2+3}{2}$, then either $U$ is affinely equivalent to the set $\{(x,x^{(p^2+1)/2}): x \in \F_{p^2}\}$, or $|D(U)| \geq \frac{p^2+p}{2}+1$.
\end{thm}

Proving direction results over $\AG(2,q)$ for general prime powers $q$ appears to be much harder. Indeed the subfield obstructions would provide a major barrier in proving extensions of Theorem~\ref{direction}; see \cite{Ball03, BBBSS} for a celebrated result of Blokhuis, Ball, Brouwer, Storme, and Sz{\H{o}}nyi. It is known that results similar to Theorem~\ref{direction} can be very helpful in studying cliques in Paley graphs and more generally Peisert-type graphs. In the literature, geometric properties of Peisert-type graphs have been explored; see \cite{AGLY22, AY22, B84, GSY23, LS83, Szi99}. Such geometrical interpretations of Peisert-type graphs turn out to play an important role in proving their EKR-type properties; we refer to Section~\ref{sec: correspondence} for a different connection, which is crucial in the proof of our stability result. 

There are also stability results in the theory of directions. Here we list two results by Sz\H{o}nyi \cite{S96}, whose proofs use tools from algebraic geometry. 

\begin{thm}[Sz\H{o}nyi \cite{S96}]\label{S1}
Let $U$ be a $D$-set of $\AG(2, q)$ consisting of $q-n$ points with $1 \leq n \leq \sqrt{q}/2$. If $|D|<(q+1)/ 2$, then $U$ can be extended to a $D$-set $V$ with $|V|=q$. Moreover, when $q$ is a prime $p$, the assumption $n \leq \sqrt{q}/2$ can be weakened to $n<(p+45)/20$.
\end{thm}
\begin{proof}
This is essentially \cite[Theorem 4]{S96}. When $q=p$ is a prime, one can use the St\"{o}hr-Voloch bound (instead of the Hasse-Weil bound) to get an improved upper bound on the number of rational points of a curve so that \cite[Theorem 4]{S96} can be refined; see \cite[Remark 3]{S96} and also \cite[Theorem 5.1]{S99}.
\end{proof}

\begin{thm}[Sz\H{o}nyi \cite{S96}]\label{S2}
Let $D \subset \PG(1,q)$ with $|D|=\frac{q+1}{2}$. Suppose that there is a $D$-set $U \subset \AG(2,q)$ such that $U$ is maximal with respect to $D$ and $|U|=q-n$ with $1\leq n \leq \sqrt{q}/2$. Then $\PG(1,q) \setminus D$ is the projection of an irreducible conic (with the form $X^2+aY^2+bZ^2+cXY+dYZ+eZX=0$ defined in the projective plane) onto the $YZ$ projective line. Moreover, when $q$ is a prime $p$, the assumption $n \leq \sqrt{q}/2$ can be weakened to $n<(p+45)/20$.
\end{thm}
\begin{proof}
This is essentially \cite[Proposition 6]{S96}. Note that in the original statement of \cite[Proposition 6]{S96}, it was implicitly assumed that the infinity direction $\infty:=[0:1]$ is in $D$ (which was assumed at the beginning of the proof in \cite[Theorem 4]{S96}); in general, one needs to slightly modify the statement. The form of the conic is explicit from the proof of \cite[Proposition 6]{S96}.

When $q=p$ is a prime, although it is not explicitly stated in \cite{S96}, one can derive the same conclusion using \cite[Remark 3]{S96} (see also the proof of Theorem~\ref{S1}). 
\end{proof}

We will call a subset $D \subset \PG(1,q)$ \emph{bad} if it satisfies the assumptions of Theorem~\ref{S2}. We define $\mathcal{B}_q$ to be the collection of bad direction sets. The following two corollaries are consequences of the above two theorems.

\begin{cor}\label{B_q}
The collection $\mathcal{B}_q$ of bad direction sets in $\PG(1,q)$ has size at most $q^5$.
\end{cor}
\begin{proof}
This follows from Theorem~\ref{S2}. Note that each $D \in \mathcal{B}_q$ arises from the projection of a conic with the form $X^2+aY^2+bZ^2+cXY+dYZ+eZX=0$, with $a,b,c,d,e \in \F_q$. Since the number of such conics is at most $q^5$ and each conic with the above form gives rise to at most one bad direction set, the statement follows. 
\end{proof}

\begin{rem}
One can get a better upper bound on $|\mathcal{B}_q|$ using the transitive property of projective general linear group $\PGL_2(\F_q)$, but it is not needed for our applications.
\end{rem}

\begin{cor}\label{stable}
Let $U$ be a $D$-set of $\AG(2, q)$ consisting of $q-n$ points with $n \leq \sqrt{q}/2$. If $|D|\leq (q+1)/ 2$ and $D \notin \mathcal{B}_q$, then $U$ can be extended to a $D$-set $V$ with $|V|=q$. Moreover, when $q=p$ is a prime, the assumption $n \leq \sqrt{q}/2$ can be weakened to $n<(p+45)/20$.
\end{cor}
\begin{proof}
When $|D|<(q+1)/2$, this follows from Theorem~\ref{S1}.

Next assume that $|D|=(q+1)/2$ and $D \notin \mathcal{B}_q$. Let $U$ be a $D$-set with the above assumption. Then Theorem~\ref{S2} implies that $U$ is not maximal with respect to $D$, that is, $U$ can be extended to a $D$-set $U'$ with $|U'|=|U|+1$, provided that $|U|<q$. We can then apply the same reasoning to $U'$ since $U'$ is still a $D$-set with $D \notin \mathcal{B}_q$. Thus, $U'$ can be extended to a $D$-set $U''$ with $|U''|=|U'|+1$, provided that $|U'|<q$. Applying the same argument repeatedly yields a $D$-set $V$ with $|V|=q$.
\end{proof}

\section{Prevalence of the strict-EKR property}\label{sec: EKR}

\begin{prop}\label{prop1}
As $q \to \infty$, almost all Peisert-type graphs of type $(\frac{q+1}{2},q)$ have the strict-EKR property. In particular, we have the EKR theorem for almost all pseudo-Paley graphs of Peisert-type.
\end{prop}

\begin{proof}
Let $q=p^n$. We need to estimate the number of graphs that fail to have the strict-EKR property. Since Cayley graphs are vertex-transitive, it suffices to consider maximum cliques containing $0$. Thus, it is equivalent to estimating the number of graphs admitting at least one non-canonical maximum clique containing $0$. 

By Theorem~\ref{subspacethm}, all non-canonical maximum cliques containing $0$ are $\F_p$-subspaces. The number of $\F_p$-subspaces of size $q$ in $\F_{q^2}$ is given by the Grassmannian
$$
\# \operatorname{Gr}(n, 2n)(\F_{p}) = \frac{(p^{2n}-1)(p^{2n}-p)\cdots (p^{2n}-p^{n-1})}{(p^{n}-1)(p^{n}-p)\cdots (p^{n}-p^{n-1})} \leq p^{2n^2}=q^{2n}.
$$

Let $V$ be an $\F_p$-subspace of size $q$ in $\F_{q^2}$, such that $V$ does not have the subfield structure in the sense that $V$ is not an $\F_q$-subspace. By Corollary~\ref{subspace_covering}, we need at least $\sqrt{q}+1$ many $\F_q^*$-cosets to cover $V\setminus \{0\}$. Let $X=\operatorname{Cay}(\F_{q^2}^+,S)$ be a Peisert-type graph of type $(\frac{q+1}{2},q)$, such that $V$ is a (maximum) clique in $X$. Then $V=V-V \subset S \cup \{0\}$. It follows that $S$ must be the union of $\frac{q+1}{2}$ many $\F_q^*$-cosets, with at least $\sqrt{q}+1$ of them having been prescribed. It follows that the number of such graphs $X$ is at most
$$
\binom{q+1-(\sqrt{q}+1)}{\frac{q+1}{2}-(\sqrt{q}+1)},
$$
which accounts for the number of graphs that fails to satisfy the strict-EKR property due to the appearance of the non-canonical maximum clique $V$.

To conclude, the density of graphs that fails to satisfy the strict-EKR property is at most 
\begin{align*}
\frac{q^{2n} \cdot \binom{q+1-(\sqrt{q}+1)}{\frac{q+1}{2}-(\sqrt{q}+1)}}{\binom{q+1}{\frac{q+1}{2}}}
&=q^{2n} \cdot \prod_{j=0}^{\sqrt{q}} \frac{\frac{q+1}{2}-j}{q+1-j}\leq \frac{q^{2n}}{2^{\sqrt{q}}}\\
&=\exp(2n \log q-\sqrt{q} \log 2)\leq \exp (2 (\log q)^2- \sqrt{q}\log 2),    
\end{align*}
which tends to $0$ as $q \to \infty$. This completes the proof.
\end{proof}

The proof of Theorem~\ref{thmnoEKR} is similar to the proof of Proposition~\ref{prop1}, but we instead consider the ``most efficient" non-canonical maximum clique in terms of the number of $\F_q^*$-cosets needed to cover it. 

\begin{proof}[Proof of Theorem~\ref{thmnoEKR}]
Let $V \subset \F_{q^2}$ be an $\F_{p^t}$-subspace of size $q$ such that $V$ is not an $\F_q$-subspace. Note that the intersection between $V$ and an $\F_q^*$-coset (together with 0) is an $\F_{p^t}$-subspace. Thus, the intersection between $V$ and each $\F_q^*$-coset has size $0$ or at least $p^t-1$. It follows that at most $\frac{q-1}{p^t-1}$ many $\F_q^*$-cosets are needed to cover $V\setminus \{0\}$. 

We consider the family of Peisert-type graphs of type $(m,q)$. Let $X=\operatorname{Cay}(\F_{q^2}^+,S)$ be a Peisert-type graph of type $(m,q)$, such that $V$ is a (maximum) clique in $X$. Note that $V$ is a non-canonical clique in $X$ provided that $S$ is the union of $m$ many $\F_q^*$-cosets, with at most $\frac{q-1}{p^t-1}$ of them having been prescribed in order to cover $V \setminus \{0\}$. It follows that the density of Peisert-type graphs of type $(m,q)$ without the strict-EKR property (due to $V$) is at least
\begin{align*}
\frac{\binom{q+1-(q-1)/(p^t-1)}{m-(q-1)/(p^t-1)}}{\binom{q+1}{m}}
&=\prod_{j=0}^{(q-1)/(p^t-1)-1} \frac{m-j}{q+1-j} \geq \bigg(\frac{m-(q-1)/(p^t-1)+1}{q+1-(q-1)/(p^t-1)+1}\bigg)^{(q-1)/(p^t-1)} \\
&=\bigg(1-\frac{q+1-m}{q+2-(q-1)/(p^t-1)}\bigg)^{(q-1)/(p^t-1)}\\
&\geq 1-\frac{q-1}{p^t-1} \cdot \frac{q+1-m}{q+2-(q-1)/(p^t-1)},
\end{align*}
where we used Bernoulli's inequality for the last step. If $q-m=o(p^t)$, then the above lower bound approaches to $1$ as $p \to \infty$. This shows that almost surely the strict-EKR property fails to hold under the given assumptions.
\end{proof}

The following corollary follows from Theorem~\ref{thmnoEKR} immediately.

\begin{cor}
As $q \to \infty$ with $q$ being a square, almost all Peisert-type graphs of type $(q-o(\sqrt{q}),q)$ do not have the strict-EKR property. 
\end{cor}

\section{Stability of canonical cliques}\label{sec: ST}
\subsection{An explicit correspondence}\label{sec: correspondence}
Let $q$ be a fixed odd prime power. We establish an explicit connection between the structure of maximum cliques of Peisert-type graphs and the direction sets determined by subsets $U \subset \AG(2,q)$ with $|U|=q$. The observation is based on the idea of realizing a Peisert-type graph projectively (as opposed to affinely, which was the idea in the previous works \cite{AGLY22, AY22}). We start by picking a fixed $v \in \F_{q^2} \setminus \F_q$ so that $\{1,v\}$ forms a basis 
of $\F_{q^2}$ over $\F_q$. This allows us to identify $\F_{q^2}$ and $\AG(2,q)$ via the embedding $\pi \colon \F_{q^2}\to \AG(2,q)$, where 
$$
\pi(a+bv)=(a,b), \quad \forall a,b \in \F_q.
$$
We also define the map $\sigma\colon \AG(2,q)\setminus \{(0,0)\} \to \PG(1,q)$ such that
$$
\sigma((a,b))=[a:b], \quad \forall (a,b) \in \AG(2,q)\setminus \{(0,0)\}.
$$

Let $X$ be a Peisert-type graph $X=\operatorname{Cay}(\F_{q^2}^+,S)$ of type $(m,q)$. We define \emph{the direction set determined by $X$} to be $\mathcal{D}(X):=\sigma(\pi(S))$. Since $S$ is the union of $m$ many $\F_q^*$-cosets in $\F_{q^2}^*$, it follows that $\mathcal{D}(X)$ is a subset of $\PG(1,q)$ with size $m$. The following proposition describes the explicit connection that will be very helpful in establishing our main results.

\begin{prop}\label{correspondence}
Let $X=\operatorname{Cay}(\F_{q^2}^+,S)$ be a Peisert-type graph  and let $C \subset \F_{q^2}$. Then we have the following correspondences:
\begin{itemize}
    \item $C$ is a clique in $X$ if and only if $D(\pi(C)) \subset \mathcal{D}(X)$, equivalently, $\pi(C)$ is a $D$-set for some $D \subset \mathcal{D}(X)$.
    \item If $|C|=q$, then $C$ is a canonical clique in $X$ if and only $|D(\pi(C))|=1$ and $D(\pi(C)) \subset \mathcal{D}(X)$. 
\end{itemize}
\end{prop}

\begin{proof}
Let $C$ be a clique. Let $x,y \in C$ with $x \neq y$. Let $\pi(x)=(a,b)$ and $\pi(y)=(c,d)$. Since $C$ is a clique, it follows that $x-y=(a-c)+(b-d)v \in S$. Therefore, the direction 
$$
[a-c:b-d]=\sigma((a-c,b-d))=\sigma(\pi(x-y)) \in \sigma(\pi(S))=\mathcal{D}(X).
$$
This shows that $D(\pi(C)) \subset \mathcal{D}(X)$. Conversely, it is also easy to verify that $D(\pi(C)) \subset \mathcal{D}(X)$ implies that $C$ is a clique.  

If $C$ is a canonical clique, then $C-C=\alpha\F_q$ for some $\alpha \in \F_{q^2}^*$, so 
$$
D(\pi(C))=\sigma(\pi((C-C) \setminus \{0\}))=\sigma(\pi(\alpha))
$$
is a singleton set. Conversely, if $|D(\pi(C))|=1$, then all points in $\pi(C)$ are collinear, and it is easy to verify that $C=\pi^{-1}(\pi(C))$ is a canonical clique.
\end{proof}

\begin{rem}
Proposition~\ref{correspondence}, together with \cite[Theorem 3.2]{AY22}, immediately imply Theorem~\ref{subspacethm}. This new proof is much simpler than the original proof in \cite[Theorem 1.2]{AY22}. 
\end{rem}

\subsection{Proof of Theorem~\ref{thmST1} and Theorem~\ref{thmST2}}

In Proposition~\ref{prop1}, we have shown that almost all Peisert-type graphs of type $(\frac{q+1}{2},q)$ have the strict-EKR property. The next proposition shows that the strict-EKR property can be strengthened to the property $\operatorname{ST}(\sqrt{q}/2)$ in the generic case.

\begin{prop}\label{EKR->ST}
Let $X$ be a Peisert-type graph of type $(\frac{q+1}{2},q)$ with the strict-EKR property. If $\mathcal{D}(X) \notin \mathcal{B}_q$, then $X$ has the property $\operatorname{ST}(\sqrt{q}/2)$. Moreover, if $q=p$ is a prime, then $X$ has the property $\operatorname{ST}(p/20)$.
\end{prop}
\begin{proof}
Let $C$ be a clique in $X$ with $|C| \geq q-\sqrt{q}/2$. To show $X$ has the property $\operatorname{ST}(\sqrt{q}/2)$, it suffices to show that $C$ is contained in a canonical clique. Let $U=\pi(C)$ and let $D=D(U)$; Proposition~\ref{correspondence} implies that $U$ is a $D$-set with $D \subset \mathcal{D}(X)$. Thus, $|D| \leq \frac{q+1}{2}$ and $D \notin \mathcal{B}_q$. It follows from Corollary~\ref{stable} that $U$ can be extended a $D$-set $V$ with $|V|=q$. In particular, by Proposition~\ref{correspondence}, $C':=\pi^{-1}(V)$ is a maximum clique in $X$ with $C \subset C'$. Since $X$ has the strict-EKR property, $C'$ must be a canonical clique. Therefore, $X$ has the property $\operatorname{ST}(\sqrt{q}/2)$. The proof for the case that $q=p$ is a prime is similar. 
\end{proof}

\begin{rem}
Let $X$ be the Paley graph of order $q^2$. We know that $X$ has the strict-EKR property \cite{B84}. If we can show that $\mathcal{D}(X) \notin \mathcal{B}_q$, then one can apply Proposition~\ref{EKR->ST} to get the first nontrivial progress towards the conjecture by Baker et. al~\cite{BEHW96} mentioned in the introduction. However, it seems more inputs from algebraic geometry are needed to understand the distribution of rational points of conics. 
\end{rem}

Now we are ready to present the proof of Theorem~\ref{thmST1}.

\begin{proof}[Proof of Theorem~\ref{thmST1}]
By Proposition~\ref{prop1}, as $q \to \infty$, almost all Peisert-type graphs of type $(\frac{q+1}{2},q)$ have the strict-EKR property. By Corollary~\ref{B_q}, as $q \to \infty$, almost all Peisert-type graphs $X$ of type $(\frac{q+1}{2},q)$ satisfies $\mathcal{D}(X) \notin \mathcal{B}_q$ since the map $X \mapsto \mathcal{D}(X)$ is injective, and $q^5/\binom{q+1}{(q+1)/2} \to 0$. Therefore, Proposition~\ref{EKR->ST} allows us to conclude that as $q \to \infty$, almost all Peisert-type graphs of type $(\frac{q+1}{2},q)$ have the property $\operatorname{ST}(\sqrt{q}/2)$.
\end{proof}

Finally, we discuss the situation where $q=p$ is a prime. In this case, we can prove stronger results with the help of Theorem~\ref{direction}.

\begin{prop}\label{prop2}
As $p \to \infty$, almost all Peisert-type graphs of type $(\lfloor \frac{2p-2}{3} \rfloor,p)$ have the strict-EKR property.
\end{prop}

\begin{proof}
Let $C$ be a non-canonical clique (containing 0) in a Peisert-type graph $X$ of type $(\lfloor \frac{2p-2}{3} \rfloor,p)$. By Proposition~\ref{correspondence}, 
$U=\pi(C)$ is a $D$-set for some $D \subset \mathcal{D}(X)$ with $|D| \geq 2$. Note that if $|D|>\frac{p+3}{2}$, then $|D|\geq \lfloor \frac{2p+1}{3} \rfloor$ by Theorem~\ref{direction}, which is impossible since $|\mathcal{D}(X)|=\lfloor \frac{2p-2}{3} \rfloor$. Therefore, in view of Theorem~\ref{direction}, we must have $|D(U)|=\frac{p+3}{2}$; furthermore, $U$ is affinely equivalent to the set $\{(x,x^{(p+1)/2}): x \in \F_p\}$. It follows that there exists an invertible matrix $M \in \GL_2(\F_p)$ such that
$$
U=M\{(x,x^{(p+1)/2}) :x \in \F_p\}.
$$
Since $|\GL_2(\F_p)|<p^4$, the number of possible direction sets $D=D(U)$ is at most $p^4$. 

As a quick summary, if $X$ is a Peisert-type graph $X$ of type $(\lfloor \frac{2p-2}{3} \rfloor,p)$ without the strict-EKR property, then it must contain a 
non-canonical clique $C$ such that $|D(\pi(C))|=\frac{p+3}{2}$ and $D(\pi(C)) \subset \mathcal{D}(X)$; furthermore, the number of possible candidates for $D(\pi(C))$ is at most $p^4$. This allows us to mimic the proof of Proposition~\ref{prop1} and conclude that the density of Peisert-type graphs $X$ of type $(m,p)$ with $m=\lfloor \frac{2p-2}{3} \rfloor$ such that the strict-EKR property fails is at most
$$
p^4 \cdot \frac{\binom{(p+1)-\frac{p+3}{2}}{m-\frac{p+3}{2}}}{\binom{p+1}{m}} < \frac{p^4}{(\frac{p+1}{m})^{(p+3)/2}}< \frac{p^4}{(\frac{3}{2})^{(p+3)/2}} \to 0
$$
as $p \to \infty$. This completes the proof.
\end{proof}

When $q=p^2$, we can also provide the following improvement on Proposition~\ref{prop1}:

\begin{prop}
As $p \to \infty$, almost all Peisert-type graphs of type $(\frac{p^2+p}{2},p^2)$ have the strict-EKR property.
\end{prop}

\begin{proof}
By Proposition~\ref{prop1}, as $p \to \infty$, almost all Peisert-type graphs of type $(\frac{p^2+1}{2},p^2)$ have the strict-EKR property. To prove the stronger result, one can mimic the proof of Proposition~\ref{prop1} and Proposition~\ref{prop2} with the help of Theorem~\ref{p^2}. We omit the detailed analysis since the proof would be very similar.
\end{proof}

We conclude the paper by proving Theorem~\ref{thmST2}.

\begin{proof}[Proof of Theorem~\ref{thmST2}]
It follows from Proposition~\ref{prop1}, Proposition~\ref{EKR->ST}, and Proposition~\ref{prop2}. The proof is similar to the proof of Theorem~\ref{thmST1}. 
\end{proof}

\section*{Acknowledgement}
The author thanks Shamil Asgarli and Venkata Raghu Tej Pantangi for many helpful discussions. 
The author is also grateful to the referees for valuable corrections and suggestions.

\bibliographystyle{abbrv}
\bibliography{main}

\end{document}